\documentclass[a4paper,12pt]{article}
\usepackage[utf8]{inputenc}
\usepackage{xcolor}
\vspace{1mm}
\usepackage{amsmath, amssymb,longtable}
\usepackage{amsthm}
\usepackage{blindtext, multicol}
\usepackage{enumerate}
\usepackage{cite}
\usepackage{setspace}
\usepackage[bookmarksnumbered, colorlinks, plainpages]{hyperref}
\usepackage{geometry}
\DeclareMathOperator{\barone}{\bar{1}^{-1}}
\DeclareMathOperator{\mx}{\mathfrak{m}}
\DeclareMathOperator{\qd}{\mathcal{Q}}
\newtheorem{theorem}{Theorem}[subsection]
\newtheorem{lemma}[theorem]{Lemma}
\newtheorem{seclemma}{Lemma}[section]

\newenvironment{mtheorem}[1]{\renewcommand\thetheorem{#1}
  \theorem
}{\endtheorem}
\newtheorem{definition}[theorem]{Definition}
\newtheorem*{note}{Note}

\newtheorem{remark}[seclemma]{Remark}
\title{Surjective Stability of Dickson-Siegel-Eichler-Roy Elementary Orthogonal Group over Projective Modules}
\author{Ambily Ambattu Asokan, Adriraj Talukdar}
\date{}
\begin{document}

\maketitle

\begin{abstract}
    Let $R$ be a commutative Noetherian ring in which $2$ is invertible. We prove that a conjugate of Petrov's odd elementary unitary group is contained in the DSER elementary orthogonal group defined over projective modules. We also show a sufficient condition regarding the Witt index of the quadratic module with a hyperbolic summand $\mathbb{H}(P)$ which implies the surjective stability of DSER orthogonal ${\rm K}_1$.
\end{abstract}
{{\it Keywords:} Quadratic modules, Odd unitary group, DSER elementary orthogonal group,  Stable range conditions, ${\rm K}_1$-group.}\\
{{\it 2020 Mathematics Subject Classification:} 19G05, 20G35, 19B14, 19B99}
\section{Introduction}
\par H. Bass introduced the notion of stable rank in the 1960s and studied the normality and stability of elementary subgroups of linear groups. Bass in \cite{Bass} and Vaserstein in \cite{VasLin} made important contribution in studying stability of the ${\rm K}_1$ functor for linear groups. In \cite{VasOS}, Vaserstein also demonstrated the ${\rm K}_1$-stability of even orthogonal and symplectic groups under stable range conditions.
\par In 1970s, A. Bak developed the notion of a parameterized stable range, namely $\Lambda$-stable rank (see \cite{Bak}), which is weaker than the stable rank of Bass. Following this, several stability results were established, notably by Bak and Tang in \cite{Herm} for Hermitian ${\rm K}_1$, by Bak, Petrov and Tang in \cite{Quad} and by Tang in \cite{Tang} for quadratic ${\rm K}_1$. All these results consider $\Lambda$-stable ranks.
\par In \cite{Petrov}, V. A. Petrov introduced the \textit{odd unitary groups}, which generalize and unify quadratic and Hermitian groups discussed in \cite{Quad} and \cite{Herm}, as well as other known classical groups. Petrov also proved the surjective stability of the \textit{odd unitary} ${\rm K}_1$ in \cite{Petrov}. The complete stability was proven by Weibo Yu in \cite{YuSt}. 
\par The \textit{DSER elementary orthogonal group} was introduced by Amit Roy in \cite{Roy} over quadratic spaces of the form $Q\!\perp \!\mathbb{H}(P)$. The normality of DSER elementary subgroup in ${\rm O}_R(Q\!\perp \!\mathbb{H}(P))$ was proven by A. A. Ambily and Ravi A. Rao in \cite{AAN}. Further  in \cite{AAK}, A. A. Ambily proved the ${\rm K}_1$-stability of this group when $Q$ and $P$ are free modules. Additionally, one can see that the $\Lambda$-stability used in \cite{AAK}, \cite{Petrov}, and \cite{YuSt} are consistent with each other. 
\par In this paper we consider the quadratic space $M = Q\!\perp \! \mathbb{H}(P)$ consisting of a projective $R$-module $P$ and a quadratic $R$-space $(Q,q)$. Note that, all projective modules in this work are finitely generated and of constant rank over a commutative Noetherian ring $R$. Aparna Pradeep and Ambily A. A. proved in \cite{comp} that when $P$ and $Q$ are free modules, Petrov's odd elementary unitary group and DSER elementary orthogonal group are congruent to each other. The novelty of this paper lies firstly in {\bf Theorem} \ref{A}, where we show that for non-free case, Petrov's group is a subgroup of the DSER group. A Local-Global principle analogous to Quillen's technique is an important ingredient of this proof.
\par Finally, we show an application of {\bf Theorem} \ref{A} by transferring certain properties from Petrov's elementary unitary group to DSER elementary orthogonal group to prove a sufficient condition regarding the Witt Index of module $M$ which implies the surjective stability of DSER odd orthogonal ${\rm K}_1$ in {\bf Theorem} \ref{B}. 
\section{Preliminaries}

\subsection{Stable Range}

\par Let $R$ be an associative ring with $1$. Let $P$ be a projective $R$-module and $P^*$ denote its dual ${\rm Hom}_R(P,R)$. 
An element $u \in P$ is said to be \textit{unimodular}, if there exists a functional $\psi \in P^*$ so that $\psi(u) = 1$. 
The projective module $P$ contains a unimodular element if and only if $P$ splits off a free summand of rank one, i.e. $P = Q\oplus Ru$ for an $R$-module $Q$ and $u$ is a unimodular element. In particular, $(a_1, \ldots, a_{n})^t \in R^{n}$ is a unimodular element if there exist elements $b_1, \ldots, b_n \in R$ such that
$\Sigma_{1=i}^n a_i b_i = 1$.
\vspace{2mm}
 \par The \textit{stable rank condition} ${\rm SR}_n$ on $R$ says that for any unimodular vector $(a_1, \ldots, a_{n+1})^t$ in $R^{n+1}$, there exist elements $b_1, \ldots, b_n \in R$ such that $(a_1 + a_{n+1}b_1, \ldots, a_n + a_{n+1}b_n)^t \in R^n$
is unimodular. The \textit{stable rank} of $R$, denoted by $sr(R)$, is the smallest number $n$ such that ${\rm SR}_n$ holds for $R$.

\vspace{2mm}

Let $R$ be a ring with $1$ and pseudoinvolution $a \mapsto \overline{a}$. Set $$\Lambda_{\min} = \{ a + \bar{a} \mid a \in R \}, 
\quad 
\Lambda_{\max} = \{ a \mid \bar{a} = a \}.$$

\par A Bak's form parameter on $R$ is an additive subgroup 
$\Lambda$ lying between $\Lambda_{\min}$ and $\Lambda_{\max}$ such that
$\bar{r}\,\Lambda\,r \subseteq \Lambda
\quad \text{for every } r \in R.$

\vspace{2mm}

\par We say the ring satisfies the \textit{$\Lambda$-stable rank condition} $\Lambda {\rm S}_n$ if $sr(R) \leq n$ and for every unimodular vector $(a_1, \ldots, a_{n+1}, b_1, \ldots, b_{n+1})^t \in R^{2n+2}$, 
there exists an $(n+1) \times (n+1)$ matrix $\beta$ with $\bar{\beta} = \bar{1} \beta \bar{1}$ and $\bar{1} \beta_{ii} \in \Lambda$, such that
$(a_1, \ldots, a_{n+1})^t + \beta (b_1, \ldots, b_{n+1})^t \in R^{n+1}$
is unimodular. The \textit{$\Lambda$-stable rank} of $R$, denoted by $sr(R, \Lambda)$, is the smallest number $n$ such that $\Lambda {\rm S}_n$ is satisfied.

\subsection{DSER Elementary Orthogonal Group} \label{sub2.2}
\par Let $R$ be a commutative ring with $1$ and $\frac{1}{2} \in R$. A projective $R$-module $M$ together with a quadratic form $q$ on $M$ is called a \textit{quadratic $R$-space}, denoted by
$(M, q)$. Let $B_q$ be the symmetric bilinear form on $M$ corresponding to $q$; that is, for $u,v \in M$, we have $B_q(u,v) = q(u+v) - q(u) - q(v)$. The quadratic form $q$ is said to be \textit{nonsingular} if the map $d_{B_q}
: M \rightarrow M^*$ induced by $d_{B_q}
(x) = B_q(x,\cdot)$, for all $x \in M$ is an isomorphism. Denote by ${\rm O}_R(M,q)$ the \textit{group of orthogonal transformations}, i.e. the group of linear automorphisms on $M$ that preserve the form $q$.
\par Let $P$ be a finitely generated projective $R$-module and $P^* = {\rm Hom}_R(P,R)$ be its dual. The module $\mathbb{H}(P) = P \oplus P^*$ equipped with the quadratic form $\rho : (x,f) \mapsto f(x)$, where $x \in P$ and $f \in P^*$ is called the \textit{hyperbolic space} of $P$. It is easy to verify that $\rho$ is a nonsingular form.

\par Let $P$ be as above and $(Q,q)$ be a non-singular quadratic $R$-space. Consider the quadratic space $M = Q\!\perp \!\mathbb{H}(P)$ with the quadratic form $q\perp \rho$. Consider $R$-linear maps $\alpha : Q \rightarrow P$ and $\beta : Q \to P^*$. Let $\alpha^t : P^* \to Q^*$ and $\beta^t : P^{**} \rightarrow Q$ be their respective transpose maps.
\par Now define $\alpha^* : P^* \to Q$ by  $\alpha^* = d^{-1}_{B_q} \circ \alpha^t$,
and $ \beta^* : P \to Q$ by $\beta^* = d^{-1}_{B_q} \circ \beta^t \circ \varepsilon$,
where $\varepsilon$ is the natural isomorphism from $P$ to $P^{**}$.

\par In \cite{Roy}, Amit Roy defined the linear transformations $E_\alpha$ and $E^*_\beta$ 
on $Q \perp \mathbb{H}(P)$ as follows and verified that they are orthogonal transformations:

\vspace{2mm}

For $z \in Q$, $x \in P$, and $f \in P^*$,
\[
E_\alpha(z,x,f) = (z - \alpha^*(f), \; x + \alpha(z) - \tfrac{1}{2}\alpha \alpha^*(f), \; f),
\]
\[
E^*_\beta(z,x,f) = (z - \beta^*(x), \; x, \; f + \beta(z) - \tfrac{1}{2}\beta \beta^*(x)).
\]

\begin{definition}
The subgroup of ${\rm O}_R(M)$ generated by $E_\alpha$ and $E^*_\beta$, for all linear maps $\alpha \in \mathrm{Hom}_R(Q,P)$ and $\beta \in \mathrm{Hom}_R(Q,P^*)$, is called 
\textit{Dickson-Siegel-Eichler-Roy (DSER) elementary orthogonal group} and is denoted by ${\rm EO}_R(Q, \mathbb{H}(P))$.
\end{definition}

The following normality result of the DSER elementary orthogonal groups was proven in \cite[Theorem 4.1]{AAN}
\begin{lemma} \label{normal}
Let $R$ be a commutative ring where $2$ is invertible. Let $P$ be a finitely generated projective $R$-module and $(Q,q)$ be a non-singular quadratic $R$-space such that rank($P$) $> 2$ and rank($Q$) $> 1$. Then the subgroup ${\rm EO}_R(Q, \mathbb{H}(P))$ is normal in ${\rm O}_R(Q\perp\mathbb{H}(P))$.
\end{lemma}

\subsection{Petrov's Odd Unitary Group}
\par V. A. Petrov introduced \text{odd unitary groups} in \cite{Petrov} that generalized many classical groups. We give a brief description here for the readers' convenience. Let $R$ be an associative ring with unity equipped with a pseudo-involution $a \mapsto \overline{a}$ for all $a \in R$. Let $M$ be a finitely generated $R$-module with an anti-Hermitian sesquilinear form $B$. Define the \textit{Heisenberg group} $\mathfrak{H} = M \times R$ with the operation ``$\dot +$" defined by,
$$(u,a) \dot + (v,b) = (u+v, a+b+B(u,v)) \text{   for } u,v \in M \text{ and } a,b \in R.$$
\par Define a right action of $R$ on $\mathfrak{H}$ by
$(u,a) \leftharpoonup b = (ub, \overline{b}\barone ab);$
and the trace map $tr : \mathfrak{H} \rightarrow R$ on $\mathfrak{H}$ by the formula
$tr((u,a)) = a - \overline{a} - B(u,u).$ Consider the following chain of subgroups of $\mathfrak{H}$: $$\{(0,a+\overline{a}) \in \mathfrak{H} | a \in R \} = \mathfrak{L}_{min} \leq \mathfrak{L} \leq \mathfrak{L}_{max} = {\rm ker}(tr).$$ 
\par A subgroup $\mathfrak{L}$ as above is called an \textit{Odd Form Parameter} if $(\mathfrak{L} \leftharpoonup R) \leq \mathfrak{L}$. We call $\qd = (B, \mathfrak{L)}$ an \textit{Odd Quadratic Form} on M. The \textit{Odd Unitary Group} $\text{U}(M,\qd)$ is the group of all isometric automorphisms $\sigma$ of $M$, which satisfy $(\sigma u - u, B(u-\sigma u,u)) \in \mathfrak{L}$ for all $u \in M$. Set $\mathfrak{L}_{ev} = \{a \in R| (0, a) \in \mathfrak{L}\}$. Then $\mathfrak{L}_{ev}$ is a form parameter in the sense of Bak \cite{Bak}.
\par For any odd quadratic space $(M, \qd)$, \textbf{Eichler-Siegel-Dickson Transvections} $T_{u,v}(a)$ are as follows : for $u,v \in M$ and $a \in R$ such that $B(u,v) = 0$ and $(u,0)\in\mathfrak{L},(v,a) \in \mathfrak{L}$,
$$T_{u,v}(a) : w \mapsto w + u\barone (B(v,w) + aB(u,w))+vB(u,w).$$
\begin{definition}
 Transvections $T_{u,v}(a)$ lie in ${\rm U}(M,\qd)$. The subgroup of ${\rm U}(M,\qd)$ generated by transvections $T_{u,v}(a)$, where $u$ is a unimodular element of $M$ is called the \textbf{Odd Elementary Unitary Group} and is denoted ${\rm EU}(M, \qd)$ following the notations of \cite{Petrov} and \cite{YuLL}.
\end{definition}
\par Let $M$ be a non-singular quadratic space. Then for any unimodular element $e_{+}$ in $M$, there exists another $e_{-} \in M$ such that $B(e_{+},e_-) = 1$. We say $(e_+,e_-)$ is a hyperbolic pair. It is easy to see that the space spanned by a hyperbolic pair is isomorphic to $\mathbb{H}(R)$. Equipped with the odd form parameter generated by $(e_+,0)$ and $(0,e_-)$, this space isometrically maps into $(M,\qd)$. \textit{Witt Index}, denoted $ind(M,\qd)$, is the greatest integer $l$ satisfying the condition that there exist $l$ mutually orthogonal hyperbolic pairs, $(e_{+1},e_{-1}), \cdots (e_{+l},e_{-l})$ in $(M,\qd)$. In that case, we can write $M = \mathbb{H}(R^l)\oplus M_0$, where $M_0$ is the orthogonal complement to the basis of the hyperbolic subspace. In this case, we denote U$(M,\qd) =$U$_{2l}(R, \mathfrak{L_0})$ and call it \textit{Odd Hyperbolic Unitary Group}. The elementary subgroup generated by $T_{e_i,e_ja}(0)$'s for $j\neq \pm i$ and $a \in R$ and $T_{e_{\pm i},v}(a)$'s for $(v,a) \in \mathfrak{L}$ is denoted ${\rm EU}_{2l}(R, \mathfrak{L_0})$ and called \textit{Odd Elementary Hyperbolic Unitary Group}. Note that, $\mathfrak{L_0} = \mathfrak{L}|_{M_0}$.

\vspace{2mm}
\par Denote by EU$_{(e_+,e_-)}(M, \qd)$ the group
generated by all elements of the form $T_{e_+,v}(a)$ and $T_{e_-,v}(a)$, where $(e_+,e_-)$ is a hyperbolic pair, $B(e_+,v) = B(e_-,v) = 0$ and
$(v,a) \in \mathfrak{L}$. It is easy to see that ${\rm EU}_{(e_+,e_-)}(M, \qd) = {\rm EU}_{2l}(R,\mathfrak{L})$.
V. A. Petrov \cite[Theorem 3 and 4]{Petrov} and Weibo Yu \cite[Corollary 3.2 and 3.3]{YuSt} proved the following result independently. 
\begin{lemma} \label{transitive}
    Let $ind(M,\qd) \geq sr(R, \mathfrak{L}_{ev})+1$. Then, for any hyperbolic pair $(u,v)$, the group ${\rm EU}_{(u,v)}(M, Q)$ acts transitively on all hyperbolic pairs in $M$. In fact, ${\rm EU}_{(u,v)}(M, \qd)$ and ${\rm EU}(M, \qd)$ are then the same  normal subgroup of ${\rm U}(M, \qd)$
\end{lemma}
\subsection{Local-Global Principle}
Let $\mathfrak{m}$ be any maximal ideal of $R$. Denote $P_{\mx} = P \otimes R_{\mx}$ for any projective $R$-module $P$ of rank $m$. $P_{\mx}$ is a free $R_{\mx}$-module. The local-global principle is a technique, notably developed independently by Quillen and Suslin, where results true for local rings or free modules can be appropriately pulled back to general rings and projective modules.
\begin{note}
    We note here the fact that finitely generated projective modules are locally free (i.e. their localization at each prime ideal is free). If the module is also finitely presented, it is projective if and only if  it is locally free. If $R$ is Noetherian, then a module is finitely generated if and only if it is finitely presented. 
\end{note}
\par We will use the localization 
$S^{-1}R$ and the polynomial ring $R[X]$ of a ring $R$ with a multiplicative system $S$ of $R$. Let $a \mapsto \bar{a}$ be the pseudo-involution 
on $R$ and let us consider only those multiplicative systems $S \subset R$ whose elements $s$ are central and satisfy the condition $\overline{s} = s\overline{1}$. Then we can define a pseudo-involution $\sigma$ on $S^{-1}R$ by 
$\sigma\!\left(\frac{a}{s}\right) = \frac{\overline{a}}{s}$. Then the natural 
homomorphism $R \to S^{-1}R$ induces a homomorphism from $U(M,\qd)$ to 
$U(S^{-1}M, S^{-1}\qd)$. Also for an indeterminate $X$, extend the pseudo-involution 
on $R$ to $R[X]$ by defining $\overline{X} = X\overline{1}$. Then the natural inclusion 
$R \to R[X]$ induces a homomorphism $(M,\qd)$ to $(M[X], \qd[X])$.
\par Similarly for any commutative ring $R$ in which $2$ is invertible and $S \subset R$ multiplicative system, we can define O$_{S^{-1}R}(S^{-1}M, S^{-1}q)$ and O$_{R[X]}(M[X], q[X])$.
 \par The following \textbf{Local-Global Principle} was proven in \cite{AAN}:
 
 \begin{lemma}\label{LG}
Let $R$ be a commutative ring with identity in which $2$ is invertible. 
Let $(Q,q)$ be a quadratic $R$-space of rank $n \geq 1$ and let $P$ be a finitely generated projective $R$-module of rank $m \geq 2$. 
Let $M = Q\! \perp \!\mathbb{H}(P)$. Assume that for every maximal ideal $\mx$ of $R$, the module $M_{\mx}$ is isomorphic to $R_{\mx}^{n+2m}$ with the canonical bilinear form. Suppose that $\theta(X) \in {\rm O}_{R[X]}(M[X])$ with $\theta(0) = I$. If $\theta_{\mx}(X) \in {\rm EO}_{R_{\mx}[X]}(Q_{\mx}[X], \mathbb{H}(R_{\mx}[X])^m)$ for all ${\mx} \in \mathrm{Max}(R)$, 
then we have $\theta(X) \in {\rm EO}_{R[X]}(Q[X], \mathbb{H}(P[X]))$.
\end{lemma}
\section{Comparison of Petrov's Odd Elementary Unitary Group and DSER Elementary Orthogonal Group}

\begin{remark} \label{remark}
    
 Let $R$ be a commutative ring where $2$ is invertible. Define on $R$ with a pseudo-involution $\bar{a} = -a$. Set the odd form parameter $\mathfrak{L} = \mathfrak{L}_{max}$ on a projective $R$-module $M$ equipped with the bilinear form $B$. Set $\qd=(B,\mathfrak{L}_{max})$. Then we observe that odd unitary transformations and orthogonal transformations coincide, i.e. $$\text{U}(M,\qd) =\text{O}_R(M,\qd).$$
Now, let $M = Q\!\perp \!\mathbb{H}(P)$ where $P$ is a projective $R$-module of rank $m$ and $(Q,q)$ is a non-singular quadratic $R$-space of rank $n$. Let $B$ be the canonical bilinear form on $M$ as defined in Subsection \ref{sub2.2}. Now let $N = \mathbb{H}(P) \perp Q$. Note that the map $\Phi: M\rightarrow N$ given by $\Phi: (z,x,f) \mapsto (x,f,z)$ (for $z\in Q$ and $(x,f) \in \mathbb{H}(P)$) is an isomorphism. We equip $N$ with the same bilinear form $B$ (we can do so because direct sum is a commutative operator in the category of finitely generated projective modules.) and define the odd quadratic form $\qd = (B, \mathfrak{L}_{max})$ on $N$. Since the bilinear forms on $M$ and $N$ are essentially same, we write $(M,\qd)$ as a quadratic space and observe that, $$\text{U}(N,\qd) =\Phi\text{O}_R(M,\qd)\Phi^{-1}$$
\end{remark}
\par The natural question that arises here is if we can compare the elementary subgroups in the respective cases. If P and Q are both free modules, a positive answer is already given. In fact an isomorphism between Petrov's odd elementary unitary group and DSER elementary orthogonal group was proven in \cite{comp}. We restate it here :

\begin{seclemma}\label{eq} 
    Let $R$ be a commutative ring with unity in which $2$ is invertible together with the pseudo-involution  $a \mapsto -a$ and let $(Q,q)$ be a free quadratic space with rank $n$ and $P$ be a free module of rank $m$. 
Then the Petrov’s odd elementary hyperbolic unitary group $\mathrm{EU}_{2m}(R, \mathfrak{L}_{max})$ (with $Q$ being the orthogonal complement of the hyperbolic submodule)
and the DSER elementary orthogonal group $\mathrm{EO}_R(Q, \mathbb{H}(P))$ are conjugate subgroups 
of the general linear group $\mathrm{GL}_{n+2m}(R)$.
\end{seclemma}
\par The conjugacy is given by the same $\Phi$ as above which in the free case acts as a permutation of the orthogonal basis. That is, we have
$$\text{EU}_{2m}(R,\mathfrak{L}_{max}) =\Phi\text{EO}_R(Q,\mathbb{H}(P))\Phi^{-1}.$$
\par Here we will show that if $P$ and $Q$ are not necessarily free, a one-sided inclusion still holds. 

\begin{mtheorem}{A} \label{A}
    Let $R$ be a commutative Noetherian ring in which $2$ is invertible. Let $P$ be any projective $R$-module of rank $m$ and $(Q,q)$ a quadratic $R$-space of rank $n$. Let $M$,$N$, $\qd$ and $\Phi$ be as given in Remark 3.1. Provided that $(e_+,e_-)$ is a hyperbolic pair in $N$, we have $$\Phi^{-1} \mathrm{EU}_{(e_+, e_-)}(N,\qd) \Phi \subset \mathrm{EO}_R (Q, \mathbb{H}(P)).$$
\end{mtheorem} \label{comp}
\begin{proof} Let $\mx \in {\rm Max}(R)$, the set of maximal ideals of $R$. Then $M_{\mx} \cong N_{\mx} \cong R_{\mx}^{2m+n}.$ Similarly $P_{\mx}[X]$ and $Q_{\mx}[X]$ are free $R_{\mx}[X]$-modules. Identify $(e_+, e_-)$ with $(e_+\otimes 1, e_- \otimes 1)$ in $N \otimes R_{\mx}[X]$.
Since $\Phi_{\mx}$ gives the same conjugacy relation as in Lemma \ref{eq}, we have
$$\text{EU}_{(e_+, e_-)}(N_{\mx}[X],\qd_{\mx}[X]) =\Phi_{\mx}\text{EO}_{R_{\mx}[X]}(Q_{\mx}[X],\mathbb{H}(P_{\mx}[X]))\Phi_{\mx}^{-1}.$$
\par Now, let $T \in {\rm EU}_{(e_+, e_-)}(N,\qd)$. Then we can construct a transformation $\sigma(X) $ in $ {\rm EU}_{(e_+, e_-)}(N[X], \qd[X])$ as following:
\par we can write $T$ as a product, $T = \prod T_{{e_{\pm 1}},v}(a)$ where each $(v,a) \in \mathfrak{L}_{max}$, which means $a = \frac{1}{2} B(v,v)$. So on $(N, \qd)$ we have,
    $$T_{e_{\pm 1},v}(a) : w \mapsto w - e_{\pm 1} B(v,w) - \frac{1}{2} B(v,v) B(e_{\pm},w)+vB(e_{\pm 1},w).$$
    Take $\sigma(X) = \prod T_{e_{\pm 1},Xv}(a)$ where,
    $$T_{e_{\pm 1},Xv}(a) : w \mapsto w - Xe_{\pm 1} B(v,w) - \frac{X^2}{2} B(v,v) B(e_{\pm},w)+XvB(e_{\pm 1},w).$$
    It is obvious that, $\sigma(0) = id_N$, the identity map on $N$ and $\sigma(1) = T$. 
\par Let $\eta(X) = \Phi^{-1} \sigma(X) \Phi \in {\rm O}_{R[X]} (M[X])$. Then clearly, $\eta(0) = id_M$, the identity on $M$ and $\eta(1) = \Phi^{-1} T \Phi$. Therefore, given any $\mx \in {\rm Max}(R)$, we have

$$\eta_{\mx}[X] \in \Phi_{\mx}^{-1}{\rm EU}_{(e_+, e_-)}(N_{\mx}[X],\qd_{\mx}[X])\Phi_{\mx} ={\rm EO}_{R_{\mx}[X]}(Q_{\mx}[X],\mathbb{H}(P_{\mx}[X])).$$

This follows from the fact that $R$ is a commutative Noetherian ring and therefore, given any two finitely generated $R$-modules $A$ and $B$, we get $({\rm Hom}_R(A,B))_{\mx} \cong {\rm Hom}_{R_{\mx}}(A_{\mx},B_{\mx})$.
\par By Lemma \ref{LG}, we have  $\eta(X) \in {\rm EO}_{R[X]} (Q[X], \mathbb{H}(P[X])$, which implies the desired result. That is, $\Phi^{-1}T\Phi \in {\rm EO}_R (Q, \mathbb{H}(P))$.
\end{proof}

\section{Stability of DSER Group} 
\par Let $R$ be a commutative Noetherian ring in which $2$ is invertible. Let $P$ be any projective $R$-module of rank $m$ and $(Q,q)$ a quadratic $R$-space of rank $n$. Let us first define the coset space,
$${\rm KO}_1(Q\!\perp \!\mathbb{H}(P)) = \frac{{\rm O}_R (Q\!\perp \!\mathbb{H}(P))}{{\rm EO}_R(Q, \mathbb{H}(P))}.$$
Now if $m\geq 2$ and $n \geq 1$, the coset space is a group by Lemma \ref{normal}. 
\par Let $e\in P$ be a unimodular element. That is, $P = Re\oplus P_0$, where $P_0$ is a projective module of rank $m-1$. By definition, there is $f \in P^*$ such that $f(e) = 1$ and $P^* = Rf \oplus P_0^*$. The inclusion ${\rm O}_R (Q\perp \mathbb{H}(P_0)) \hookrightarrow {\rm O}_R (Q\!\perp \!\mathbb{H}(P))$ induces a map $\Psi : {\rm KO}_1(Q\perp \mathbb{H}(P_0)) \rightarrow {\rm KO}_1(Q\!\perp \!\mathbb{H}(P))$ such that for $\sigma \in {\rm O}_R(Q\perp \mathbb{H}(P_0))$,
$$\Psi : \sigma{\rm EO}_R(Q, \mathbb{H}(P_0)) \longmapsto \sigma {\rm EO}_R(Q, \mathbb{H}(P)).$$
We want to find out whether and when this map $\Psi$ is surjective, injective or an isomorphism. The following result for when $P$ and $Q$ are free modules was proven in \cite{AAK}. The result assumes a $0$-stable range condition.
\begin{seclemma} \label{isom}
Let $R$ be a commutative ring of $sr(R, 0)= l$ in which $2$ is invertible. 
Then, for all $m \geq l+1$, the canonical map
\[
KO_{1}(Q \perp \mathbb{H}(R)^r) \;\longrightarrow\; KO_{1}(Q \perp \mathbb{H}(R)^m)
\]
is surjective for $l \leq r < m$, and when $m \geq l+2$, the canonical homomorphism
\[
KO_{1}(Q \perp \mathbb{H}(R)^{m-1}) \;\longrightarrow\; KO_{1}(Q \perp \mathbb{H}(R)^m)
\]
is an isomorphism.
\end{seclemma}

We want to generalize the above result for finitely generated projective modules. 
\begin{note}[X] \label{X}
    For any inclusion of groups $G < H$ such that $C \unlhd G $, $D \unlhd H$ and $C<D$; we can show by basic group theory that the canonical map $G/C \longrightarrow H/D$ is surjective if and only if $H = DG$. Additionally, the above map is injective if and only if $C = G \cap D$.
    \par So to prove surjectivity and injectivity of $\Psi$, it is enough to show that $${\rm O}_R(Q, \mathbb{H}(P)) = {\rm EO}_R(Q, \mathbb{H}(P))\cdot {\rm O}_R(Q, \mathbb{H}(P_0)) \text{, and}$$ $${\rm EO}_R(Q, \mathbb{H}(P_0)) = {\rm EO}_R(Q, \mathbb{H}(P))\cap {\rm O}_R(Q, \mathbb{H}(P_0)) \text{ respectively.}$$
\end{note}

    We will prove a sufficient condition for surjectivity of the map $\Psi$. This will begin by recalling $M$, $N$, $\qd$ and $\Phi$ in Remark \ref{remark}. Since the odd form parameter on $N$ is $\mathfrak{L} = \mathfrak{L}_{max} = \{(v, \frac{1}{2}B(v,v):v \in N\}$, we have $\mathfrak{L}_{ev} = 0$. This will allow us to use Lemma \ref{transitive} on $N$. The goal is to transfer the transitive action of ${\rm EU}(N, \qd)$ to ${\rm EO}(Q \perp \mathbb{H}(P))$, using Theorem \ref{A}.

\begin{mtheorem}{B} \label{B}
    Let $R$ be a commutative Noetherian ring where $2$ is invertible. Let $P$ be any projective $R$-module of rank $m \geq 2$ and $(Q,q)$ a quadratic $R$-space of rank $n \geq 1$. Let $(e,f)$ be a hyperbolic pair such that $P = Re \oplus P_0$ and $P^* = Rf \oplus P_0^*$. Denoting, $M = Q \perp \mathbb{H}(P)$, let $(M, \qd)$ be as Remark \ref{remark}. Further, assume the Witt index, $ind (M, \qd) = j$ and $sr(R, 0) = l$.
    \par If $j \geq l+1$, the map $\Psi : {\rm KO}_1(Q\perp \mathbb{H}(P_0)) \longrightarrow {\rm KO}_1(Q\!\perp \!\mathbb{H}(P))$ is surjective.
\end{mtheorem}
\begin{proof}
    We will prove the theorem in two steps\\
    \textbf{Step 1:} Let $N$ and $\Phi$ be as in Remark \ref{remark}. Then $(\Phi e, \Phi f)$ is a hyperbolic pair in $N$. Let $(u,v)$ be another hyperbolic pair in $M$. Since $j = ind(N,\qd) \geq l$, by Lemma \ref{transitive}, there exist a $\sigma \in {\rm EU}_{(\Phi e, \Phi f)}(N, \qd)$ such that $\sigma(\Phi e) = \Phi u$ and $\sigma(\Phi f) = \Phi v$.
    Clearly then, $$\Phi^{-1} \sigma \Phi e=u \text{ and } \Phi^{-1} \sigma \Phi f = v.$$
    By Theorem \ref{A}, we know $\Phi^{-1} \sigma \Phi \in {\rm EO}_R(Q, \mathbb{H}(P))$. Therefore, ${\rm EO}_R(Q, \mathbb{H}(P))$ acts transitively on the hyperbolic pairs of $M$.\\
    \textbf{Step 2:} It is trivial to observe
    $${\rm EO}_R(Q, \mathbb{H}(P)) \cdot {\rm O}_R(Q, \mathbb{H}(P_0)) \subset {\rm O}_R(Q, \mathbb{H}(P)) $$
    Now, let $\gamma \in {\rm O}_R(Q\!\perp \!\mathbb{H}(P))$. Then $B(\gamma e,\gamma f) = B(e,f) = 1$, i.e. $(\gamma e,\gamma f)$ is a hyperbolic pair. By \textbf{Step 1}, we have a transformation $\eta \in {\rm EO}_R(Q, \mathbb{H}(P))$ such that, $\eta \gamma e = e$ and $\eta \gamma f = f$. Therefore, $\eta \gamma$ fixes $e$ and $f$. That is, $\eta \gamma \in {\rm O}_R(Q, \mathbb{H}(P_0))$. So, finally we have
    $$\gamma = \eta^{-1} (\eta \gamma) \in {\rm EO}_R(Q, \mathbb{H}(P))\cdot {\rm O}_R(Q, \mathbb{H}(P_0)).$$
    Following the observation in Note \ref{X}, this concludes the proof.
\end{proof}

\begin{remark}
    We have found a sufficient condition for surjective stability of the orthogonal ${\rm K}_1$-functor. Essentially the result above shows that if a projective module contains a sufficiently large free submodule, surjective ${\rm K}_1$-stability of DSER elementary orthogonal group over it is preserved. However, we are yet to establish an equivalent generalization of the injective stability found on free modules. In this effort, applying the Local-Global principle requires a homotopy of DSER transformations strictly lying inside ${\rm O}_{R[X]}(Q[X]\perp \mathbb{H}(P_0[X])$. The main obstruction is proving the existence of such a homotopy.
\end{remark}

\section{Acknowledgement} Both authors thank the support received from the ANRF SURE Scheme (SUR/ 2022/ 004894), Government of India. In addition, the second author thanks Dr VK Aparna Pradeep for her important suggestions.

\end{document}